\documentclass{article}
\usepackage{amsthm,amsmath,amssymb,url}
\newtheorem{thm}{Theorem}[section]
\newtheorem{prop}[thm]{Proposition}
\newtheorem{cor}[thm]{Corollary}

\newtheorem{defn}[thm]{Definition}
\newtheorem{rem}[thm]{Remark}
\numberwithin{equation}{section}
\newtheorem{example}[thm]{Example}

\usepackage{centernot}
\usepackage{mathtools}
\DeclareMathOperator{\Aut}{Aut}
\makeatletter
\DeclareRobustCommand{\Div}{\mathrel{\Div@}}
\DeclareRobustCommand{\nDiv}{\mathrel{\nDiv@}}

\newcommand{\Div@}{%
  \mkern2mu\nonscript\mkern-2mu 
  \mathpalette\Div@@\relax
  \mkern2mu\nonscript\mkern-2mu
}
\newcommand{\Div@@}[2]{%
  \hbox{%
    \sbox\z@{$#1T$}%
    \vbox to \ht\z@{%
      \offinterlineskip\m@th
      \hbox{$#1.$}\vfil
      \hbox{$#1.$}\vfil
      \hbox{$#1.$}%
    }%
  }%
}
\newcommand{\nDiv@}{\centernot\Div@}
\makeatother

\begin{document}
\title{Almost abelian numbers}
\author{Iulia C\u at\u alina Ple\c sca\footnote{
Faculty of Mathematics of "Al. I. Cuza" University of Ia\c si, Romania, 
e-mail: dankemath@yahoo.com
ORCID: 0000-0001-7140-844X}\text{ } and Marius T\u arn\u auceanu\footnote{Faculty of Mathematics of
"Al. I. Cuza" University of Ia\c si, Romania,
e-mail: tarnauc@uaic.ro}}

\maketitle

\begin{abstract}
In this article we introduce the concept of almost $\mathcal{P}$-numbers. We survey the existing results in literature for almost cyclic numbers and give characterizations for almost abelian and almost nilpotent numbers proving these two are equivalent.
\end{abstract}


\textbf{Classification}: primary: 20D05, secondary: 20D40


\section{Introduction}
Throughout this article, let $\mathcal{P}$ be a class of groups. We will denote the dihedral group of order $n$ by $D_n$ and the cyclic group of order $n$ by $C_n$. For more standard notations and definitions see \cite{Isaacs}. 
\begin{defn}\label{def:P}
A positive integer $n$ is called a \textbf{$\mathcal{P}$-number} if all groups of order $n$ are in class $\mathcal{P}$.
\end{defn}
Some of the most obvious particular cases for this definition (cyclic, abelian, nilpotent) use the concept of nilpotent factorization. 
\begin{defn}\cite{new}
A positive integer $n=p_{1}^{n_{1}} \cdots p_{k}^{n_{k}}, p_{i}$ distinct primes, is said to have \textbf{nilpotent factorization} if $p_{i}^{l} \not \equiv 1 \bmod p_{j}$ for all positive integers $i, j$ and $l$ with $1 \leq l \leq n_{i}$. 
\end{defn}
We have the following characterizations.
\begin{prop}\cite{new}
A positive integer $n$ is 
\begin{itemize}
\item
a nilpotent number if and only if it has nilpotent factorization.
\item
an abelian number if and only if it is a cube–free number with nilpotent factorization.
\item
a cyclic number if and only if it is a square–free
number with nilpotent factorization
\end{itemize} 
\end{prop}
The characterization of cyclic numbers can also be given as follows:
\begin{prop}\cite{new}
A positive integer $n$ is cyclic if and only if the number and its Euler totient function $\varphi(n)$ are coprime, that is
\[\gcd(n,\varphi(n))=1.\]
\end{prop}
\begin{example}
All prime numbers are cyclic numbers.
\end{example}
If we loosen the hypothesis in Definition \ref{def:P}, we reach the following definition.
\begin{defn}
A positive integer $n$ is called an \textbf{almost $\mathcal{P}$-number} if all but one group of order $n$ (up to isomorphism) are in class $\mathcal{P}$.
\end{defn}
This definition has started from the following result given in \cite{Master}:
\begin{thm}
Let $G$ be a group of order $n = p_1^{n_1}\dots p_k^{n_k}$, where $p_1<p_2<\dots<p_k$. Then there are exactly two groups (up to isomorphism) of order $n$ if and only if $k\geq 2$ and one of the following scenarios occurs:
\begin{equation}\label{sit1}
\begin{aligned}
&\text{for all } l \in\{1,\dots,k\}:  n_l = 1 \text{, and }\\&\text{there exists a unique pair } (i,j)\in\{1,\dots,k\}^2\text{  such that }p_i|p_j-1
\end{aligned}
\end{equation}
or
\begin{equation}\label{sit2}
\begin{aligned}
&\text{there exists a unique } j\in\{1,\dots,k\} \text{ such that }n_j = 2, \text{ and }\\
&\text{for all } l\in\{1,\dots,k\}\setminus \{j\}: n_l=1, \text{ and }\\
&\text{there exists a unique } i\in \{1,\dots,k\}\setminus \{j\}:p_i\mid p_j-1, \ p_i\nmid p_j + 1, \text{ and }\\
&\text{for all } (\alpha, \beta)\in\{1,\dots,k\}^2 \setminus \{(i,j)\}: p_\alpha \nmid p_\beta - 1 
\end{aligned}
\end{equation}
\end{thm}
The result above is equivalent to the following:
\begin{cor}\label{cor}
A positive integer $n=p_1^{n_1}\dots p_n^{n_k}$ is almost cyclic if $k\geq 2$ and either \eqref{sit1} or \eqref{sit2} hold. 
\end{cor}
In what follows, we introduce similar results to Corollary \ref{cor} for abelian groups and nilpotent groups. 

\section{Main results}

\begin{thm}\label{theorem}
A positive integer $n=p_1^{n_1}\dots p_n^{n_k}$ is almost abelian if and only if either \eqref{sit1} or
\begin{equation}\label{sit2'}
\begin{aligned}
&\text{there exists a unique } j\in\{1,\dots,k\} \text{ such that }n_j = 2, \text{ and }\\
&\text{for all } l\in\{1,\dots,k\}\setminus \{j\}: n_l=1, \text{ and }\\
&\text{there exists a unique } i\in \{1,\dots,k\}\setminus \{j\}:p_i\mid p_j+1, \text{ and }\\
&\text{for all } (\alpha, \beta)\in\{1,\dots,k\}^2 \setminus \{(i,j)\}: p_\alpha \nmid p_\beta - 1 
\end{aligned}
\end{equation}
\end{thm}
\begin{proof}
"$\Rightarrow$" Let us suppose that $n$ is almost abelian. It follows that:
\begin{equation} \label{divisor}
\text{ For any divisor }d\text{ of }n: d|n,\text{ there is at most one nonabelian group of order }d.
\end{equation}
In addition, since $n$ is not abelian, there are two possibilities:
\begin{itemize}
\item
$\text{there exists } r\in\{1,\dots,k\}$ such that $n_r\geq 3$.
This contradicts \eqref{divisor}, since there are at least two nonabelian groups of order $p_r^{n_r}$ when $n_r\geq 3$. 
\item
\begin{equation}\label{ij}
\text{there exist } i, j\in\{1,\dots,k\} \text{ such that }p_i|p_j^{n_j}-1
\end{equation}
and 
\begin{equation} n_r\leq 2, \text{for all } r\in\{1,..,k\}
\end{equation}
Let us analyse which of the exponents can be $2$.
From \eqref{ij}, it follows that there exists a nontrivial semi-direct product $C_{p_j}^{n_j}\rtimes C_{p_i}$. 
\begin{itemize}
\item
If $n_r=2$, for $r\in\{1,2,\dots,k\}\setminus\{i,j\}$, it follows that there are two nonabelian groups of order $p_ip_j^{n_j}p_r^2$:
\[ ((C_{p_j})^{n_j}\rtimes C_{p_i})\times (C_{p_r})^2 \text{ and }((C_{p_j})^{n_j} \rtimes C_{p_i})\times C_{p_r^2}\]
which gives us a contradiction.
\item 
If $n_i=2$ then we will have two nonabelian groups of order $p_i^2p_j^{n_j}$:
\[(C_{p_j})^{n_j} \rtimes C_{p_i^2} \text{ and } ((C_{p_j})^{n_j}\rtimes C_{p_i})\times C_{p_i},\]
which again gives a contradiction. 
\end{itemize}
It follows that $n_r=1, \text{for all } r\neq j$.\\
If in \eqref{ij} there are two distinct pairs $(i,j)$ and $(\alpha,\beta)$ such that $p_{i}|p_j^{n_j}-1$ and $p_{\alpha}|p_{\beta}^{n_\beta}-1,$ then it follows that there are two distinct nonabelian groups of order $p_i p_j^{n_j}p_{\alpha}p_{\beta}^{n_{\beta}}$:
\[((C_{p_{\beta}})^{n_{\beta}}\rtimes C_{p_{\alpha}})\times C_{p_{i}p_{j}^{n_{j}}} \text{ and } C_{p_{\alpha} p_{\beta}^{n_{\beta}}}\times ((C_{p_{j}})^{n_j}\rtimes C_{p_i})\]
 which proves that there is at most one pair .\\
If $n_j=1$, \eqref{sit1} is satisfied.\\
If $n_j=2$, let us assume that 
\begin{equation}
p_i|p_j-1. 
\end{equation}
It follows that there are two distinct nonabelian groups of order $p_ip_j^2$:
\[ (C_{p_j})^2\rtimes C_{p_i} \text{ and } C_{{p_j}^2}\rtimes C_{p_i}\]
which contradicts our hypothesis. It follows that $p_i\nmid p_j-1\xRightarrow{p_i|p_j^2-1} p_i|p_j+1$ which gives \eqref{sit2'}.
\end{itemize}
"$\Leftarrow$" We will prove by \textbf{induction over $n$ }that:
\begin{multline}\label{induction}
\text{ If }n\text{ has at least two non-prime factors and satisfies either \eqref{sit1} or \eqref{sit2'}},\\
\text{it follows that there is a unique nonabelian group of order }n.
\end{multline} 
The \textbf{base case} is $n=6$. There is just one nonabelian group of order $6$: the symmetric group $S_3$.\\
Let us proceed to the \textbf{inductive step}. Assume that \eqref{induction} holds for any positive integer with at least two factors $n'<n$. Let $G$ be a nonabelian group of order $n$. We can assume that $p_1<p_2<\dots p_k$. It follows that $j\geq 2$. Indeed, if $j=1$, we have two possibilities:
\begin{itemize}
\item
$p_i|p_1-1$ which implies $p_i<p_1$, a contradiction.
\item
$p_i|p_1+1\xRightarrow {p_i>p_1} p_i=p_1+1\Rightarrow p_1=2$ and $p_i=p_2=3\Rightarrow p_j|p_i-1$ which is a contradiction.
\end{itemize}
Thus $j\geq 2$ and therefore $n_1=1$, hence the $p_1$-Sylow subgroups of $G$ are cyclic of order $p_1$. Thus, by the Burnside normal p-complement theorem, $G$ has a $p_1$-normal complement, i.e. $\text{there exists } H\triangleleft G$ with $|H|=n/p_1$ such that $G=H\rtimes C_{p_1}$. We identify two cases:
\begin{itemize}
\item 
$n$ satisfies \eqref{sit1}.
\begin{itemize}
\item
If $i=1$, it follows that $H$ is cyclic and we get:\\
$G=C_{\frac{n}{p_1}}\rtimes C_{p_1}\cong (C_{p_j}\rtimes C_{p_1})\times C_{\frac{n}{p_1p_j}}$
\item
If $i\geq 2$, there are again two possibilities:
\begin{itemize}
\item
$H$ abelian $\Rightarrow H$ cyclic which is analogous with above.
\item
$H$ nonabelian $\Rightarrow H\cong (C_{p_j}\rtimes C_{p_i})\times C_{\frac{n}{p_1p_ip_j}}\Rightarrow$\\$\begin{aligned}G&\cong ((C_{p_j}\rtimes C_{p_i})\times C_{\frac{n}{p_1p_ip_j}})\rtimes C_{p_1}\\&\cong ((C_{p_j}\rtimes C_{p_i})\rtimes C_{p_1})\times C_{\frac{n}{p_1p_ip_j}}\end{aligned}$.\\
Since $|\Aut(C_{p_j}\rtimes C_{p_i})|=p_j(p_j-1)$ and $p_1\nmid p_j(p_j-1)$ the semidirect product $(C_{p_j}\rtimes C_{p_i})\rtimes C_{p_1}$ is trivial therefore
$$\begin{aligned}G&\cong ((C_{p_j}\rtimes C_{p_i})\times C_{p_1})\times C_{\frac{n}{p_1p_ip_j}}\\&\cong (C_{p_j}\rtimes C_{p_i})\times C_{\frac{n}{p_ip_j}}.\end{aligned}$$
It follows that in this case there is a single nonabelian group of order $n$ which is:
$(C_{p_j}\rtimes C_{p_i})\times C_{\frac{n}{p_ip_j}}$.
\end{itemize} 
\end{itemize}
\item
$n$ checks \eqref{sit2'}.
\begin{itemize}
\item
If $i=1$, then $H$ is abelian and we get:
We get two possibilities:
\begin{itemize}
\item
$H\cong C_{p_j^2}\times C_{\frac{n}{p_1p_j^2}}$. It follows that $G\cong (C_{p_j^2}\rtimes C_{p_1})\times C_{\frac{n}{p_1p_j^2}}$.\\
Since $p_1\nmid |\Aut((C_{p_j^2})|=p_j(p_j-1)$, it follows that $C_{p_j^2}\rtimes C_{p_1}=C_{p_j^2}\times C_{p_1}$, therefore $G\cong C_n$ which is a contradiction.
\item
$H\cong (C_{p_j})^2\times C_{\frac{n}{p_1p_j^2}}$ thus $G\cong ((C_{p_j})^2\rtimes C_{p_1})\times C_{\frac{n}{p_1p_j^2}}$ which is nonabelian. 
\end{itemize} 
\item
If $i\geq 2$, then we again have two cases:
\begin{itemize}
\item
$H\cong ((C_{p_j})^2\rtimes C_{p_i})\times C_{\frac{n}{p_1p_ip_j^2}}$, then:
$$G\cong (((C_{p_j})^2\rtimes C_{p_i})\rtimes C_{p_1})\times C_{\frac{n}{p_1p_ip_j^2}}.$$
From \cite{Automorphisms}, it follows that $|\Aut((C_{p_j})^2\rtimes C_{p_i}|=2(p_j^2-1)p_j^2\nDiv p_1,$ therefore:
$$\begin{aligned}
G&\cong (((C_{p_j})^2\rtimes C_{p_i})\times C_{p_1})\times C_{\frac{n}{p_1p_ip_j^2}}\\ &\cong ((C_{p_j})^2\rtimes C_{p_i})\times C_{\frac{n}{p_ip_j^2}}.\end{aligned}.$$ 
\item $H$ abelian. We identify two possibilities:
\begin{itemize}
\item
$H\cong C_{p_2}\times\dots\times C_{p_j}\times \dots\times C_{p_k}\cong C_{\frac{n}{p_1}}\Rightarrow G$ cyclic which is false.
\item
$H\cong C_{p_2}\times \dots C_{p_j}^2\times \dots \times C_{p_k}\cong (C_{p_j})^2\times C_{\frac{n}{p_1p_j^2}}$\\
Let us observe that 
\begin{equation}\label{eq}
p_1\nmid |\Aut((C_{p_j})^2 \times C_{\frac{n}{p_1p_j^2}})|=(p_j^2-1)(p_j^2-p_j)\cdot \varphi\left(\frac{n}{p_1p_j^2}\right).
\end{equation}
$G\cong ((C_{p_j})^2\times C_{\frac{n}{p_1p_j^2}})\rtimes C_{p_1}\cong((C_{p_j})^2\times C_{\frac{n}{p_1p_j^2}})\times C_{p_1}\stackrel{\eqref{eq}}{\cong} (C_{p_j})^2\times C_{\frac{n}{p_j^2}}$.\\ This means $G$ is abelian which is false.
\end{itemize}
\end{itemize}
Therefore, also in this case there is only one nonabelian group of order $n$.
\end{itemize}
\end{itemize}
\end{proof}
\begin{cor}
Let $n=p_1^{n_1}\cdot\dots\cdot p_k^{n_k}$ where $2=p_1<\dots<p_k$. It follows that $n$ is almost abelian if and only if $k=2$ and $n_1=n_2=1$.
\end{cor}
\begin{cor}
Let $n=p_1\cdot\dots\cdot p_k$, where $p_1<\dots<p_k$. Then $n$ is almost abelian if and only if there is a unique pair $(i,j)\in\{1,\dots,k\}^2$ such that $p_i|p_j-1$.
\end{cor}
\begin{rem}
If $n$ is almost cyclic, $n$ is either abelian or almost abelian.
The converse is false. For example, $75$ is almost abelian, but $75$ is not almost cyclic/cyclic.
\end{rem}
\begin{thm}
A number $n=p_1^{n_1}\cdot\dots\cdot p_k^{n_k}$ is almost nilpotent if and only if $n$ is almost abelian, i.e. $k\geq 2$ and $n$ checks \eqref{sit1} or \eqref{sit2'}.  
\end{thm}
\begin{proof}
"$\Leftarrow$" The converse follows from Theorem \ref{theorem} since all the groups constructed in the proof are non-nilpotent.\\
"$\Rightarrow$"Let us assume $n$ is almost nilpotent. It follows that $\text{for all } d|n, \text{there exists } $ at most one non-nilpotent group of order $d$. Since $n$ is non-nilpotent, it follows that there is $(i,j)\in\{1,\dots,k\}^2$ and $1\leq d_j\leq n_j$ such that $p_i|p_j^{d_j}-1$. It follows that $\alpha_j=n_j$. Otherwise there would be two non-nilpotent non-isomorphic groups of order $p_j^{n_j}p_i$:$$((C_{p_j})^{\alpha_j}\rtimes C_{p_i})\times C_{p_j^{n_j-d_j}} \text{ and }(C_{p_j})^{n_j}\rtimes C_{p_i}.$$
Furthermore, the pair $(i,j)$ is unique. Otherwise, if there were two pairs $(i',j')\neq (i,j)$ such that $p_{i'}|p_{j'}^{n_j}-1$, again there would be two non-nilpotent non-isomorphic groups of order $n$:
$$((C_{p_j})^{n_j}\rtimes C_{p_i})\times C_{\frac{n}{p_ip_j^{n_j}}} \text{ and }((C_{p_{j'}})^{n_{j'}}\rtimes C_{p_{i'}})\times (C_{\frac{n}{p_{i'}p_{j'})^{n_{j'}}}}.$$
Let us observe that $n_r=1, \text{for all } r\neq i, j$. Indeed, otherwise there would exist at least two distinct groups $P_r$ and $Q_r$ of order $p_r^{n_r}$, which would give two non-nilpotent, non-isomorphic groups of order $p_j^{n_j}p_ip_r^{n_r}$:
$$((C_{p_j})^{n_j}\rtimes C_{p_i})\times P_r \text{ and }((C_{p_j})^{n_j}\rtimes C_{p_i})\times Q_r.$$
Analogously, we can show that $n_i=1$.\\
If $n_j=1$, \eqref{sit1} holds.\\
If $n_j\geq 2$, then $n_j=2$, since otherwise we would two non-nilpotent, non-isomorphic groups of order $p_j^{n_j}p_i$. Thus 
\begin{equation}\label{pi}
p_i|p_j^2-1.
\end{equation} In addition, if $p_1|p_j-1$, there are two non-nilpotent, non-isomorphic groups of order $p_j^{2}p_i$:
$$(C_{p_j})^2\rtimes C_{p_i}\text{ and }C_{p_j^2}\rtimes C_{p_i}.$$
It follows that $p_i\nmid p_j-1\xRightarrow{\eqref{pi}} p_i|p_j+1$, therefore we get \eqref{sit2'}, which concludes the proof.
\end{proof}


\section*{Funding}
The authors did not receive support from any organization for the submitted work.

\section*{Conflicts}
The authors declare that they have no conflict of interest.

\section*{Acknowledgements} The authors are grateful to the reviewers for their remarks which improved the previous version of the paper.

\end{document}